\documentclass[12pt]{article}
\usepackage{latexsym,amssymb,amsmath,amsthm,enumerate,geometry,float,cite,tikz,pgfplots}
\usepackage{verbatim}
\newtheorem{theorem}{Theorem}
\newtheorem{proposition}[theorem]{Proposition}
\newtheorem{corollary}[theorem]{Corollary}
\usepackage{lineno}
\usepackage{setspace}

\begin{document}
\onehalfspace

\title{Irregularity of Graphs respecting Degree Bounds}
\author{
Dieter Rautenbach\and Florian Werner}
\date{}

\maketitle
\vspace{-10mm}
\begin{center}
{\small 
Institute of Optimization and Operations Research, Ulm University, Ulm, Germany\\
\texttt{$\{$dieter.rautenbach,florian.werner$\}$@uni-ulm.de}
}
\end{center}

\begin{abstract}
Albertson defined the irregularity of a graph $G$ as
$irr(G)=\sum\limits_{uv\in E(G)}|d_G(u)-d_G(v)|$.
For a graph $G$ with $n$ vertices, $m$ edges, maximum degree $\Delta$,
and $d=\left\lfloor \frac{\Delta m}{\Delta n-m}\right\rfloor$, we show 
$$irr(G)\leq d(d+1)n+\frac{1}{\Delta}\left(\Delta^2-(2d+1)\Delta-d^2-d\right)m.$$
{\bf Keywords:} Irregularity; Mostar index
\end{abstract}

\section{Introduction}

In the present paper we study the {\it irregularity} $irr(G)$ 
of finite, simple, and undirected graphs $G$,
which was defined by Albertson \cite{al} as
$$irr(G)=\sum\limits_{uv\in E(G)}|d_G(u)-d_G(v)|,$$
where $d_G(u)$ denotes the degree of a vertex $u$ in $G$
and $E(G)$ denotes the edge set of $G$.
For a graph $G$ of order $n$ 
with the property that $irr(G')\leq irr(G)$
whenever $G'$ is obtained from $G$ by exactly one edge addition or deletion,
Albertson \cite{al} showed that $G$ is isomorphic to a split graph $S_{p,n-p}$,
which arises from the disjoint union of a clique $C$ of order $p$
and an independent set $I$ of order $n-p$ 
by adding all possible edges between $C$ and $I$.
From this he deduced
\begin{eqnarray}\label{eb1}
irr(G) & < & \frac{4n^3}{27},
\end{eqnarray}
which is asymptotically best possible.
Abdo, Cohen, and Dimitrov \cite{abcodi} refined (\ref{eb1}) 
to $irr(G)\leq irr(S_{p,n-p})$ for $p=\left\lfloor\frac{n}{3}\right\rfloor$.
Hansen and M\'{e}lot \cite{hame} proved a best possible upper bound 
on the irregularity of a graph with a given number $n$ of vertices and $m$ of edges;
in fact, the extremal graphs arise from some $S_{p,n-p}$ 
by adding edges that are all incident with the same vertex.
For a graph $G$ with $n$ vertices, $m$ edges, maximum degree $\Delta$, and minimum degree $\delta$,
Zhou and Luo \cite{zhlu} proved the two bounds
\begin{eqnarray}
irr(G) & \leq & m\sqrt{\frac{2n\Big(2m+(n-1)(\Delta-\delta)\Big)}{n+\Delta-\delta}-4m}\,\,\,\,\,\,\,\,\,\mbox{ and }\label{eb2}\\
irr(G) & \leq & \sqrt{m\Big(2mn(\Delta+\delta)-n^2\Delta\delta-4m^2\Big)},\label{eb3}
\end{eqnarray}
which can be satisfied with equality only if $G$ is regular or a complete bipartite graph.
Using variations of $S_{p,n-p}$, 
Abdo et al.~\cite{abcodi} provided 
lower bounds on the maximum irregularity of graphs 
of given order, maximum degree, and minimum degree.
The maximum irregularity of 
bipartite graphs \cite{hera},
graphs of bounded clique number \cite{zhlu}, and 
graphs with a given number of vertices of degree $1$ \cite{dobudaho,lichhuzh}
was also studied.

Our contribution in the present paper is the following best possible bound 
on the irregularity of graphs depending on their order, size, and maximum degree.

\begin{theorem}\label{theorem1}
Let $G$ be a graph with $n$ vertices, $m$ edges, and maximum degree at most $\Delta$,
where $\Delta$ is a positive integer.
If $d\in \{0,\ldots,\Delta-1\}$ is such that $\frac{2m}{n}\in \left[\frac{2\Delta d}{\Delta+d},\frac{2\Delta (d+1)}{\Delta+d+1}\right]$, then
\begin{eqnarray}\label{e1}
irr(G) \leq d(d+1)n+\frac{1}{\Delta}\left(\Delta^2-(2d+1)\Delta-d^2-d\right)m.
\end{eqnarray}
\end{theorem}
The complete bipartite graphs $K_{\Delta,d}$ and $K_{\Delta,d+1}$ 
as well as disjoint unions of these graphs satisfy (\ref{e1}) with equality,
which implies that Theorem \ref{theorem1} is best possible
up to terms of lower order that come from divisibility issues.
For given values of $n$, $\Delta$, and $m$, 
a suitable value of $d$ within Theorem \ref{theorem1} is 
$d=\left\lfloor \frac{\Delta m}{\Delta n-m}\right\rfloor$.
Eliminating $d$ yields the following smooth version of 
the piecewise affine bound (\ref{e1}).

\begin{corollary}\label{corollary1}
If $G$ is a graph with $n$ vertices, $m$ edges, and maximum degree at most $\Delta$,
where $\Delta$ is a positive integer, then 
$$irr(G) \leq \frac{(\Delta n-2m)\Delta m}{\Delta n-m}
<\left(3-2\sqrt{2}\right)\Delta^2n.$$
\end{corollary}
Figure \ref{fig1} 
illustrates the bounds on Theorem \ref{theorem1} and Corollary \ref{corollary1}.

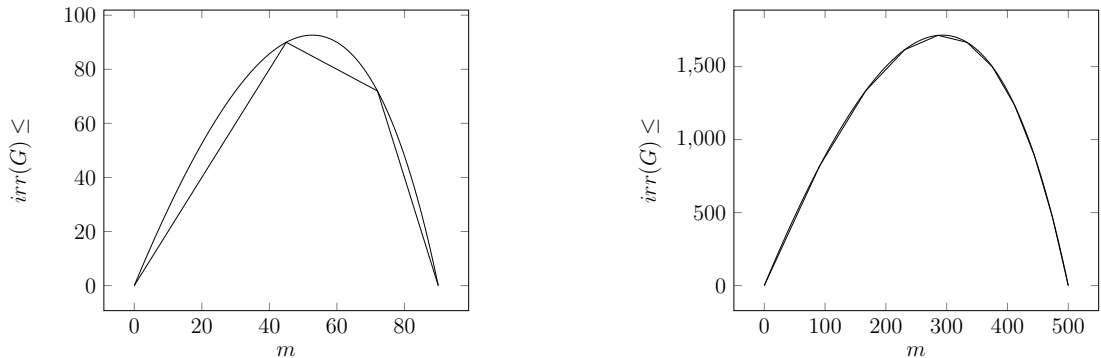
\begin{figure}[H]
    \centering
    \begin{tikzpicture}[scale=0.7]
    \def\n{60}
    \def\D{3}
    \begin{axis}[
    ylabel=$irr(G)\leq$,
    xlabel=$m$,
    y label style={at={(-0.05,0.5)}},
    samples=\D*\n/2+1,domain=0:\D*\n/2]
    \addplot[thin, mark=none ]plot (\x, {(floor(\D*\x/ (\D*\n-\x)))*((floor(\D*\x/ (\D*\n-\x)))+1)*\n+1/\D*(\D*\D-(2*(floor(\D*\x/ (\D*\n-\x)))+1)*\D-(floor(\D*\x/ (\D*\n-\x)))*(floor(\D*\x/ (\D*\n-\x)))-(floor(\D*\x/ (\D*\n-\x))))*\x});
    \addplot[thin, mark=none ]plot (\x,{((\D*\n-2*\x)*\D*\x)/(\D*\n-\x)});
    \end{axis}
    \end{tikzpicture}\hspace{2cm}
    \begin{tikzpicture}[scale=0.7]
    \def\n{100}
    \def\D{10}
    \begin{axis}[
    ylabel=$irr(G)\leq$,
    xlabel=$m$,
    y label style={at={(-0.05,0.5)}},
    samples=\D*\n/2+1,domain=0:\D*\n/2]
    \addplot[thin, mark=none ]plot (\x, {(floor(\D*\x/ (\D*\n-\x)))*((floor(\D*\x/ (\D*\n-\x)))+1)*\n+1/\D*(\D*\D-(2*(floor(\D*\x/ (\D*\n-\x)))+1)*\D-(floor(\D*\x/ (\D*\n-\x)))*(floor(\D*\x/ (\D*\n-\x)))-(floor(\D*\x/ (\D*\n-\x))))*\x});
    \addplot[thin, mark=none ]plot (\x,{((\D*\n-2*\x)*\D*\x)/(\D*\n-\x)});
    \end{axis}
    \end{tikzpicture}
\caption{The bound from (\ref{e1}) and the bound $\frac{(\Delta n-2m)\Delta m}{\Delta n-m}$ for $(n,\Delta)=(60,3)$ on the left and $(n,\Delta)=(100,10)$ on the right.}\label{fig1}
\end{figure}
Figure \ref{fig2} compares the bound from Corollary \ref{corollary1}
with the bounds provided by Zhou and Luo \cite{zhlu}.
\begin{figure}[H]
    \centering
    \begin{tikzpicture}[scale=0.7]
    \def\n{60}
    \def\D{10}
    \begin{axis}[
    ylabel=$irr(G)\leq$,
    xlabel=$m$,
    y label style={at={(-0.05,0.5)}},
    samples=\D*\n/2+1,domain=0:\D*\n/2]
    \addplot[thin, mark=none ]plot (\x,{\x*sqrt(2*\n*(2*\x+(\n-1)*\D)/(\n+\D)-4*\x)});
    \addplot[thin, mark=none ]plot (\x,{sqrt(\x*(2*\x*\n*\D-4*\x*\x))});
    \addplot[thin, mark=none ]plot (\x,{((\D*\n-2*\x)*\D*\x)/(\D*\n-\x)});
    \end{axis}
    \end{tikzpicture}
\caption{The bounds (\ref{eb2}) (top), (\ref{eb3}) (middle), and 
$\frac{(\Delta n-2m)\Delta m}{\Delta n-m}$ (bottom) for $(n,\Delta,\delta)=(60,10,0)$.
Note that a minimum degree of exactly $0$ can easily be obtained by considering extremal configurations on $n-1$ vertices and adding one isolated vertex; 
that is, requiring minimum degree $0$ should have little influence on the bounds.}\label{fig2}
\end{figure}
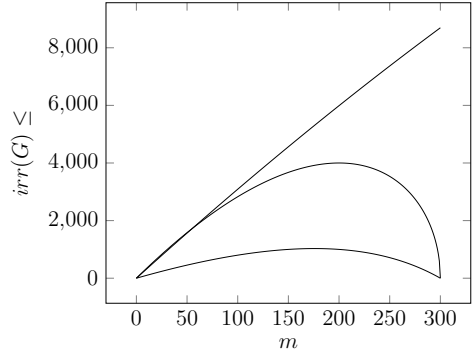
The proof of Theorem \ref{theorem1} is surprisingly simple and elegant.
It is based on a linear programming approach,
which we developed in connection with the Mostar index \cite{miparawe1,miparawe2},
a parameter that is closely related to the irregularity \cite{aldo,gaxudo,gets}.
Both parameters display a similar tradeoff:
In order to obtain a large value of the parameter, 
one needs many edges, but, if there are too many edges, 
then the average contribution per edge goes down.
Fortunately, for the setting we consider in this paper, 
linear constraints are sufficiently strong to correctly capture this tradeoff.

Exploiting this approach also leads to results involving minimum degree conditions.

For integers $\Delta>\delta\geq 0$, let 
$$\delta^*={\rm arg max}\left\{ \frac{\Delta\left(\Delta-i\right)i}{\Delta+i}:i\in \{ \delta,\ldots,\Delta\}\right\}.$$
It is easy to see that
$$\delta^*\in \Big\{
\left\lfloor\left(\sqrt{2}-1\right)\Delta\right\rfloor,
\left\lceil\left(\sqrt{2}-1\right)\Delta\right\rceil\Big\}$$
if $\delta\leq \left\lfloor\left(\sqrt{2}-1\right)\Delta\right\rfloor$ 
and $\delta^*=\delta$ otherwise.

\begin{proposition}\label{proposition1}
Let $G$ be a graph with $n$ vertices, 
maximum degree at most $\Delta$, and minimum degree at least $\delta$,
where $\Delta>\delta\geq 0$ are integers.
If $\delta^*$ is as above, then
\begin{eqnarray}\label{e2}
irr(G) \leq 
\frac{\Delta\left(\Delta-\delta^*\right)\delta^*}{\Delta+\delta^*}n.
\end{eqnarray}
\end{proposition}
The graph $K_{\Delta,\delta^*}$ as well as disjoint unions of copies of this graph
satisfy (\ref{e2}) with equality.

\begin{proposition}\label{proposition2}
Let $G$ be a graph with $n$ vertices, $m$ edges,
maximum degree at most $\Delta$, 
and minimum degree at least $\delta$,
where $\Delta>\delta\geq 1$ are integers.
If $\frac{2m}{n}\in \left[\delta,\frac{2\Delta \delta}{\Delta+\delta}\right]$, then
\begin{eqnarray}\label{e3}
irr(G) \leq & 2\Delta m-\delta\Delta n.
\end{eqnarray}
\end{proposition}
The graph $K_{\Delta,\delta}$ and $\delta$-regular graphs
as well as disjoint unions of these graphs satisfy (\ref{e3}) with equality.
For an average degree $\frac{2m}{n}$ bigger than $\frac{2\Delta \delta}{\Delta+\delta}$,
the bound (\ref{e1}) is best possible for graphs of minimum degree at least $\delta$.

All proofs are given in the next section.

\section{Proofs}

Let $G$ be a graph with $n$ vertices, $m$ edges, and maximum degree $\Delta$,
where $\Delta$ is a positive integer.
Let $I_0=\{ 0,1,\ldots,\Delta\}$ and $I=I_0\setminus \{ 0\}$.
The irregularity of $G$ is at most the optimum value ${\rm OPT}(P)$
of the following linear programm:
$$\begin{array}{rrrcll}
& \max  & \sum\limits_{i,j\in I:i<j}(j-i)m_{i,j}&&&\\
& s.th.  & \sum\limits_{i\in I_0}n_i&=&n,&\\
(P) \,\,\,\,\,\,\,\,\, & & \sum\limits_{i\in I}in_i&=&2m,&\\
& & 2m_{i,i}+\sum\limits_{j\in I:j<i}m_{j,i}+\sum\limits_{j\in I:j>i}m_{i,j}-in_i &=&0&
\mbox{ for every $i\in I$,}\\
&   & n_i & \in & \mathbb{R}_{\geq 0}& \mbox{ for every $i\in I_0$, and}\\
&   & m_{i,j} & \in & \mathbb{R}_{\geq 0}& \mbox{ for every $i,j\in I$ with $i\leq j$}.
\end{array}$$
The variables $n_i$ and $m_{i,j}$ within $(P)$ correspond 
to the number of vertices of $G$ of degree $i$ and 
the number of edges $uv$ of $G$ with $\{ d_G(u),d_G(v)\}=\{ i,j\}$,
respectively.

The dual of $(P)$ is the following linear programm:
$$\begin{array}{rrrcll}
& \min  &  nx+2my &&&\\
& s.th.  & z_i+z_j &\geq &j-i& \mbox{ for every $i,j\in I$ with $i<j$},\\
(D) \,\,\,\,\,\,\,\,\, & & x+iy& \geq &iz_i & \mbox{ for every $i\in I$},\\
&   & x & \in & \mathbb{R}_{\geq 0}, &\\
&   & y & \in & \mathbb{R}, &\mbox{ and}\\
&   & z_i & \in & \mathbb{R}_{\geq 0} &  \mbox{ for every $i\in I$}.
\end{array}$$
For our argument, we only need weak duality 
${\rm OPT}(P)\leq {\rm OPT}(D)$ between $(P)$ and $(D)$,
captured by the following inequality chain:
\begin{eqnarray*}
&&\sum\limits_{i,j\in I:i<j}(j-i)m_{i,j}\\ & \leq & 
\sum\limits_{i,j\in I:i<j}(\underbrace{z_i+z_j}_{\geq j-i})m_{i,j}
+\underbrace{2\sum\limits_{i\in I}z_im_{i,i}
+xn_0
+\sum\limits_{i\in I}(x+iy-iz_i)n_i}_{\geq 0}\\
&=& 
\left(\sum\limits_{i\in I_0}n_i\right)x
+\left(\sum\limits_{i\in I}in_i\right)y
+\sum\limits_{i\in I}\left(2m_{i,i}+\sum\limits_{j\in I:j<i}m_{j,i}+\sum\limits_{j\in I:j>i}m_{i,j}-in_i\right)z_i\\
&=& nx+2my.
\end{eqnarray*}

\begin{proof}[Proof of Theorem \ref{theorem1}]
Let $G$ be as in the statement of Theorem \ref{theorem1}.

We claim that $\left(x,y,(z_i)_{i\in I}\right)$ with
\begin{eqnarray*}
x &=& d(d+1),\\
y &=& \frac{1}{2\Delta}\left(\Delta^2-(2d+1)\Delta-d^2-d\right),\mbox{ and}\\
z_i&=& \frac{1}{i}x+y\mbox{ for $i\in I$}
\end{eqnarray*}
is a feasible solution for $(D)$.

Clearly, we have $x\geq 0$ and $x+iy\geq iz_i$ for every $i\in I$. 
Since $z_i$ is decreasing in $i$,
we obtain
$$z_i\geq z_{\Delta}
=\frac{1}{2\Delta}(\Delta-1-d)(\Delta-d)\geq 0\mbox{ for every $i\in I$}.$$
Furthermore, the constraint 
$z_i+z_j \geq j-i$ is satisfied for every $i,j\in I$ with $i<j$ 
if and only if 
$z_i+z_\Delta \geq \Delta-i$ for every $i\in I\setminus \{ \Delta\}$, 
which --- using the values of $x$, $y$, and the $z_i$ --- is equivalent to the true statement
\begin{eqnarray*}
z_i+z_\Delta-(\Delta-i)=\frac{(d+1-i)(d-i)}{i}\geq 0.
\end{eqnarray*}
Altogether, it follows that $\left(x,y,(z_i)_{i\in I}\right)$ is a feasible solution of $(D)$.

Now,
\begin{eqnarray*}
irr(G) &\leq & {\rm OPT}(P)\\
&\leq& {\rm OPT}(D)\\
& \leq & nx+2my\\
&=& d(d+1)n+\frac{1}{\Delta}\left(\Delta^2-(2d+1)\Delta-d^2-d\right)m,
\end{eqnarray*}
which completes the proof.
\end{proof}

\begin{proof}[Proof of Corollary \ref{corollary1}]
The definition of $d$ in the statement of Theorem \ref{theorem1} implies 
$d\in D$ for the interval $D=\left[\frac{\Delta m}{\Delta n-m}-1,\frac{\Delta m}{\Delta n-m}\right]$.
By Theorem \ref{theorem1}, the function 
$$f:D\to \mathbb{R}:x\mapsto x(x+1)n+\frac{1}{\Delta}\left(\Delta^2-(2x+1)\Delta-x^2-x\right)m$$
satisfies 
$irr(G)\leq f(d)=\underbrace{\left(n-\frac{m}{\Delta}\right)}_{>0}d^2
-\left(2m-n+\frac{m}{\Delta}\right)d+m(\Delta-1)$.

Since $f$ is a quadratic polynomial in $d$ with positive coefficient for $d^2$,
we obtain
$$irr(G)\leq \max\left\{ f(x):x\in D\right\}=\max\left\{f\left(\frac{\Delta m}{\Delta n-m}-1\right),f\left(\frac{\Delta m}{\Delta n-m}\right)\right\}
=\frac{(\Delta n-2m)\Delta m}{\Delta n-m},$$
in fact, the two extreme points of the interval $D$ yield the same $f$-value.

Considered as a function of $m$, the term $\frac{(\Delta n-2m)\Delta m}{\Delta n-m}$
is maximized for
$m^*=\frac{\Delta n}{2+\sqrt{2}}$, which implies 
$$irr(G)\leq \frac{(\Delta n-2m)\Delta m}{\Delta n-m}
<\left(3-2\sqrt{2}\right)\Delta^2 n.$$
Note that the final inequality is strict because $m^*$ is not an integer.
\end{proof}

\begin{proof}[Proof of Proposition \ref{proposition1}]
Let $G$ be as in the statement of Proposition \ref{proposition1}
and let $G$ have $m$ edges.

Let $I_0=\{ \delta,\ldots,\Delta\}$ and $I=I_0\setminus \{ 0\}$,
and consider $(P)$ and $(D)$ as above 
for these modified index sets capturing the degree constraints.

We claim that $\left(x,y,(z_i)_{i\in I}\right)$ with
\begin{eqnarray*}
x &=& \frac{\Delta\left(\Delta-\delta^*\right)\delta^*}{\Delta+\delta^*},\\
y &=& 0,\mbox{ and}\\
z_i&=& \frac{1}{i}x\mbox{ for $i\in I$}
\end{eqnarray*}
is a feasible solution for $(D)$.
Clearly, $x$ and the $z_i$ are non-negative.
Again, since $z_i$ is decreasing in $i$, the constraint 
$z_i+z_j \geq j-i$ is satisfied for every $i,j\in I$ with $i<j$ if and only if 
$z_i+z_\Delta \geq \Delta-i$ for every $i\in I\setminus \{ \Delta\}$, 
which 
--- using the definition of $\delta^*$ as well as the values of $x$, $y$, and the $z_i$ --- 
is equivalent to the true statement
$x\geq \frac{\Delta\left(\Delta-i\right)i}{\Delta+i}.$
As for Theorem \ref{theorem1}, this completes the proof.
\end{proof}

\begin{proof}[Proof of Proposition \ref{proposition2}]
Let $G$ be as in the statement of Proposition \ref{proposition2}.

Let $I_0=I=\{ \delta,\ldots,\Delta\}$
and consider $(P)$ as above 
for these modified index sets capturing the degree constraints.
Since there is no more variable ``$n_0$'', 
within the dual $(D)$ of $(P)$,
the constraint ``$x\in \mathbb{R}_{\geq 0}$''
is replaced with the constraint ``$x\in \mathbb{R}$'',
that is, the variable $x$ may now assume negative values.

We claim that $\left(x,y,(z_i)_{i\in I}\right)$ with
\begin{eqnarray*}
x &=& -\delta\Delta,\\
y &=& \Delta,\mbox{ and}\\
z_i&=& \frac{1}{i}x+y\mbox{ for $i\in I$}
\end{eqnarray*}
is a feasible solution for $(D)$.

Clearly, $z_i=\left(1-\frac{\delta}{i}\right)\Delta$ is non-negative for $i\geq \delta$.
Since $z_i$ is increasing in $i$, the constraint 
$z_i+z_j \geq j-i$ is satisfied for every $i,j\in I$ with $i<j$ if and only if 
$z_\delta+z_j \geq j-\delta$ for every $j\in I\setminus \{ \delta\}$, 
which is equivalent to the true statement
$$z_\delta+z_j-(j-\delta)=\frac{(j-\delta)(\Delta-j)}{j}\geq 0.$$
As for Theorem \ref{theorem1}, this completes the proof.
\end{proof}
In many cases, 
the extremal graphs that we described for our results 
can be made connected by edge swaps,
that is, by replacing edges $uv$ and $u'v'$ in different components 
with the edges $uv'$ and $u'v$.
It seems possible to characterize all extremal graphs 
exploiting the complementary slackness conditions for $(P)$ and $(D)$.

\end{document}